\documentclass[11pt,a4paper,final]{article}

\newif\ifsiamart
\makeatletter%
\@ifclassloaded{siamart171218}%
  {\siamarttrue}%
  {\siamartfalse}%
\makeatother%

\ifsiamart\else
\usepackage{authblk}
\fi
\usepackage{silence}
\usepackage[T1]{fontenc}
\usepackage[utf8]{inputenc}
\usepackage{csquotes}
\usepackage[UKenglish]{babel}
\usepackage{paralist}
\usepackage[shortlabels]{enumitem}
\usepackage[margin=.6in]{geometry}
\usepackage{xcolor}
\usepackage{graphicx}
\usepackage{bbm}
\usepackage{amsmath}
\usepackage{amssymb}
\usepackage{amsfonts}
\usepackage{stmaryrd}
\usepackage{mathrsfs}
\usepackage{mathtools}
\usepackage{epstopdf}
\usepackage{comment}

\usepackage{xspace}
\usepackage{float}
\usepackage{stackengine}

\ifsiamart\else
\usepackage{amsthm}
\usepackage{hyperref}
\hypersetup{%
    linktocpage=true,
    colorlinks=true,
    citecolor=red,
    linkcolor=blue,
    urlcolor=blue,
}
\usepackage[nameinlink,capitalise]{cleveref}
\newcommand{\email}[1]{\href{mailto:#1}{#1}}

\fi

\usepackage{setspace}
\usepackage{xcolor}
\usepackage{color}
\usepackage{mathtools, thmtools,amssymb}
\usepackage[normalem]{ulem}
\usepackage{mathrsfs}
\usepackage{array}
\usepackage{hyperref}
\usepackage{balance}

\DeclareMathOperator{\supp}{supp}
\def\A{\mathcal A}

\def\B{\mathcal B}

\def\C{\mathcal C}

\def\D{\mathcal D}
\def\Ds{\mathscr{D}}
\def\Rs{\mathscr{R}}
\def\DD{\bar{\mathcal{D}}}

\def\Hc{\mathcal H}
\def\K{\mathcal K}

\def\L{\mathcal L}

\def\M{\mathcal M}

\def\V{\mathcal V}

\def\R{\mathbb R}

\newcommand{\X}{\mathcal{X}}
\def\Xs{\mathscr X}

\newcommand{\wy}{w_y}
\newcommand{\Y}{\mathcal{Y}}
\def\Ys{\mathscr Y}

\newcommand{\U}{\mathcal{U}}
\def\Us{\mathscr U}

\newcommand{\blkdiag}{\text{blkdiag}}

\newtheorem{theorem}{Theorem}

\newtheorem{remark}{Remark}

\newtheorem{definition}{Definition}

\newtheorem{example}{Example}

\DeclareMathOperator{\Id}{I}

\DeclareMathOperator*{\argmin}{arg\,min}

\renewcommand{\d}{\mathrm{d}}

\newcommand{\norm}[1]{\lVert#1\rVert}

\newcommand{\lec}[1]{\textcolor{teal}{{#1}}}

\def\<#1,#2>{\langle #1,\,#2\rangle}

\newcommand{\pCLMs}{\theta}
\newcommand{\pk}{\vartheta}
\newcommand{\px}{\theta_x}
\newcommand{\pkx}{\vartheta_x}
\newcommand{\pu}{\theta_u}
\newcommand{\pku}{\vartheta_u}
\newcommand{\pxx}{\theta_{xx}}
\newcommand{\pux}{\theta_{ux}}
\newcommand{\pxy}{\theta_{xy}}
\newcommand{\puy}{\theta_{uy}}

\relax

\newif\ifextended
\extendedtrue
\ifextended

\else

\fi

\usepackage{glossaries}
\newacronym{MPC}{MPC}{model predictive control}
\newacronym{SLS}{SLS}{system level synthesis}
\newacronym{SLP}{SLP}{system level parameterization}
\newacronym{CLM}{CLM}{closed-loop map}
\newacronym{PDE}{PDE}{partial differential equation}

\usepackage[url=true,backend=biber,style=trad-abbrv,url=false,doi=false,isbn=false]{biblatex}
\DeclareFieldFormat{volume}{volume \textbf{#1}}
\DeclareFieldFormat[article]{volume}{\textbf{#1}}

\ifsiamart\else
\let\oldparagraph=\paragraph
\renewcommand\paragraph[1]{\oldparagraph{#1.}}
\fi

\begin{document}
\title{Convex Constrained Controller Synthesis for Evolution Equations}

\ifsiamart
    \author{%
    }
\else
    \author[1]{Lauren Conger}
    \author[2]{Antoine P. Leeman}
    \author[1]{Franca Hoffmann}
    \affil[ ]{\footnotesize \email{lconger@caltech.edu},
        \email{aleeman@ethz.ch},
    \email{franca.hoffmann@caltech.edu}}
    \affil[1]{\footnotesize Department of Computing and Mathematical Sciences, Caltech}
    \affil[2]{\footnotesize Institute for Dynamic Systems and Control, ETH Z\"urich}
    \date{}
\fi

\maketitle

\begin{abstract}
We propose a convex controller synthesis framework for a large class of constrained linear systems, including those described by (deterministic and stochastic) partial differential equations and integral equations, commonly used in fluid dynamics, thermo-mechanical systems, quantum control, or transportation networks. 
Most existing control techniques rely on a (finite-dimensional) discrete description of the system, via ordinary differential equations. Here, we work instead with more general (infinite-dimensional) Hilbert spaces.
This enables the discretization to be applied after the optimization (optimize-then-discretize).
Using output-feedback System Level Synthesis (SLS), we formulate the controller synthesis as a convex optimization problem. 
Structural constraints like sensor and communication delays, and locality constraints, are incorporated while preserving convexity, allowing parallel implementation and extending key SLS properties to infinite dimensions.
The proposed approach and its benefits are demonstrated on a linear Boltzmann equation.
\end{abstract}
\glsresetall 

\section{Introduction}
Many dynamical systems are described by evolution equations and have inherently continuous state spaces, such as fluid dynamics and plasma behavior in fusion reactors \cite{cattafesta_actuators_2011}, heat transfer \cite{manzoni_quadratic_2021}, aerospace applications \cite{joslin_aircraft_1998}, and bacterial movement \cite{perthame_transport_2007}.
For control of wind farms \cite{shapiro_windfarm}, considering the wind velocity as the state, one can maximize the power.
In these setting, computing controllers is particularly challenging due to the infinite-dimensional nature of the state.
One set of current control techniques largely relies on a discretization of the system, resulting in an ordinary differential equation (ODE) \cite{Hinze2012}. 
Although this traditional technique has proven useful for many evolution systems, the discretization generates a large number of states, and leads to an optimization problem often too costly to solve.
\begin{figure}[!ht]
    \centering
    \includegraphics[trim= 141 40 178 5,clip, width=0.6\linewidth]{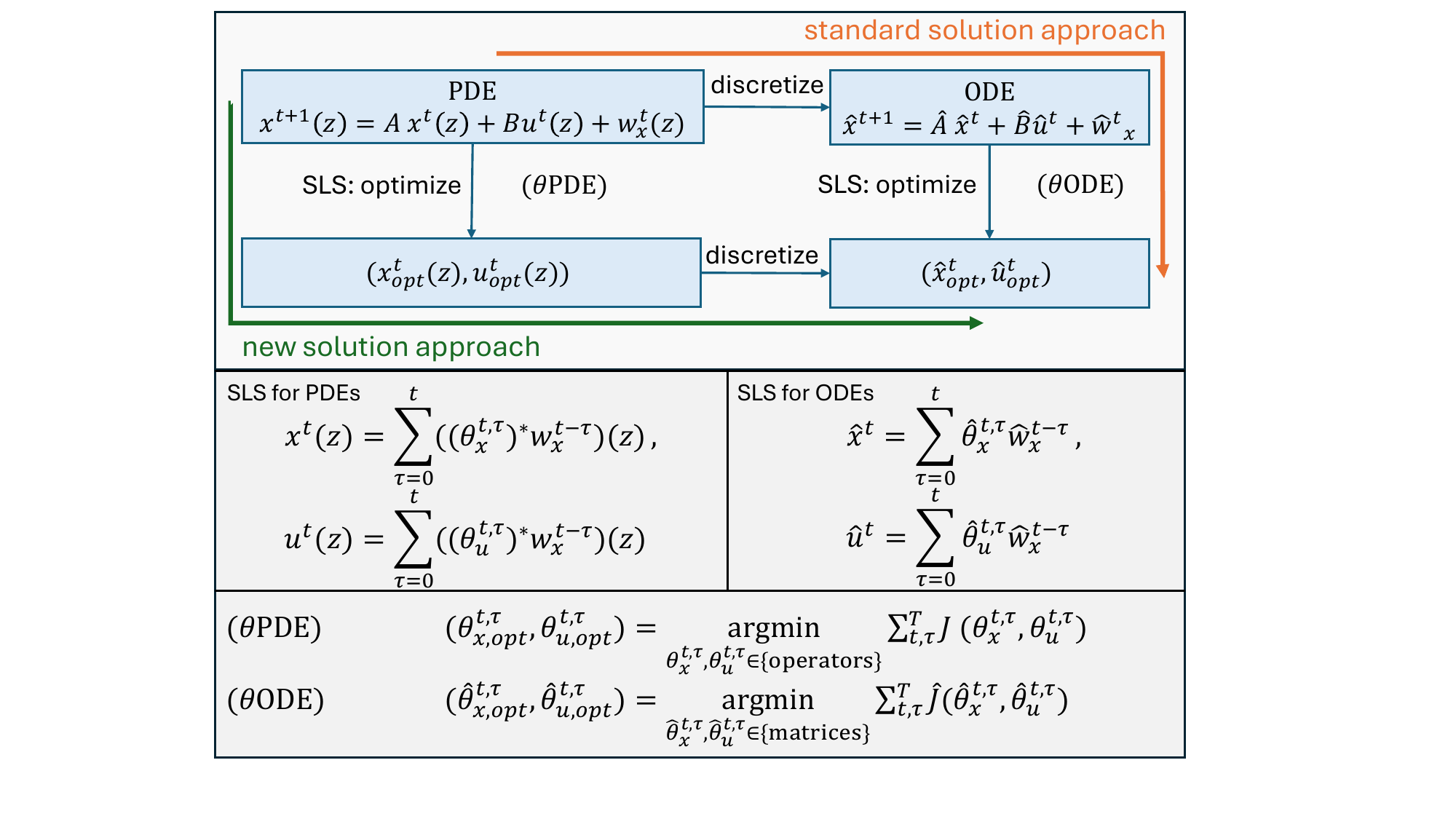}
    \caption{Traditional control approaches discretize PDE dynamics and then solve a finite-dimensional optimal control problem (discretize-then-optimize). We propose optimizing in the infinite-dimensional space, and then discretizing the solution (optimize-then-discretize). This requires an infinite-dimensional \gls{SLS} framework.}
    \label{fig:optimize_discretize_diagram}
\end{figure}
A second set of methods to solve evolution equation control problems includes techniques used for boundary control \cite{smyshlyaev_adaptive_2010}. Recent progress in this area leverages 
\textit{backstepping}, in which a change of variables allows for a closed-form solution of the boundary controller \cite{ascencio_backstepping_2018}, e.g., for specific dynamics \cite{troltzsch2010optimal} such as the reaction-diffusion equation \cite{ayamou_finite-dimensional_2024,si_boundary_2018,vazquez_boundary_2019}.
Beyond boundary, or system-specific control, a third, older set of results surrounds more general control in Hilbert spaces. 
 While classical work exists on adaptive control \cite{wen_robust_1989}, robust control \cite{venkatesh_robust_2000}, controllability \cite{slemrod_note_1974}, and stabilizing operators \cite{gibson_riccati_1979},
 practical tools are still needed to bridge the gap between theory and applications.
 
To address some of the challenges in these three sets of approaches to \gls{PDE} control, we present a convex controller synthesis method for a large class of linear systems, including \glspl{PDE} and integral equations, using ideas from system level synthesis (\gls{SLS}), a popular robust and distributed control tool \cite{wang_separable_2018, chen2024robust, anderson2019system, LEEMAN2024_fastSLS, dean2020robust}, and generalize results from \cite{jensen_topics_2020,bamieh_distributed_2002} by formulating solutions in the weak sense (broadening the classes of evolution equations and solutions) and numerically considering the setting of non-spatially invariant systems.
\gls{SLS} 
was recently developed to parameterize state- and output-feedback controllers as a convex optimization problem, allowing structural constraints while preserving convexity \cite{anderson2019system,wang_separable_2018,conger_characterizing_2024}.
Our method mitigates discretization issues by optimizing directly on infinite-dimensional Hilbert spaces; in general, optimizing and discretizing do not necessarily commute \cite{liu_non-commutative_2019}, and instead of optimizing over a large discrete state space, we use an optimize-then-discretize approach (see Figure~\ref{fig:optimize_discretize_diagram}). This requires conceptualization of the \gls{SLS} framework in infinite dimensions, pushing beyond the current finite-dimensional theory
and serving both more accurate theoretical analysis and expected practical performance gains. 

\textit{Contributions:} We propose a convex (output-feedback) controller parametrization for deterministic or stochastic evolution equations in infinite-dimensional spaces, incorporating convex structural constraints such as locality, sensor delay, and communication delays. 
We prove that optimizing over convex closed-loop maps is equivalent to optimizing over linear feedback controllers, which may not be convex.
On a numerical implementation using an integral equation, we show how orthonormal basis functions reduce the problem to a finite-dimensional optimization of coefficients, with subproblems solved in parallel.
This work lays the foundation for several natural generalizations, including time-varying operators in the dynamics and \gls{SLS} for continuous-time evolution equations.

\section{Problem Setup}
\subsection{Notation}
Let $\<\cdot,\cdot>_\Hc$ denote the inner product in a real Hilbert space $\Hc$ and its topological dual $\Hc^*:=\{f:\Hc\to\R \text{ continuous, linear functions}\}$. For a linear bounded operator $A:\Ds(A)\subseteq \Hc_1\to \Hc_2$ between two Hilbert spaces with domain $\Ds(A)$ and range $\Rs(A)$, its dual operator $A^*:\Ds(A^*)\subseteq\Hc_2^* \to \Hc_1^*$ maps $f(\cdot)\in\Hc_2^*$ to the function $g(\cdot)=f(A\cdot)\in \Hc_1^*$. The Riesz map allows us to identify $A^*$ with the adjoint $A':\Ds(A')\subseteq\Hc_2\to\Hc_1$ of $A$, which is defined via the relation $\<x_2,Ax_1>_{\Hc_2}=\<A'x_2,x_1>_{\Hc_1}$ for all $x_1\in\Ds(A)\subseteq\Hc_1, x_2\in \Ds(A')\subseteq\Hc_2$.  Thanks to the Riesz map, $\Hc$ is isometrically isomorphic to $\Hc^*$. From now on, we identify $\Hc^*$ with $\Hc$ and $A'$ with $A^*$. We remark that the operator domain is part of the definition of the operator itself, and therefore $A$ and $\Ds(A)$ are always stated together.

Throughout, we will use that the state $x^t\in\X$, input $u^t\in\U$, output $y^t\in\Y$, disturbance $w_x^t\in\X$, and observation noise $w_y^t\in\Y$ at time $t$ are in Hilbert spaces. 
For a Hilbert space $\X$ and discrete times $t\in\{0,\dots,T\}$, we use the script $\Xs \coloneqq \X^{\otimes(T+1)}$ to denote the $T+1$ dimensional vector of elements $x^0,\dots,x^T\in\X$. This is again a Hilbert space when endowed with the inner product $\<x,\bar x>_\Xs:=\sum_{t=0}^T\<x^t,\bar x^t>_\X$. In the context of \glspl{PDE}, a very useful Hilbert space is the Sobolev space $H^k(\Omega)$ with $\Omega\subseteq\R^d$ denoting the class of functions whose weak derivatives exist up to order $k$ and are in $L^2(\Omega):=\{f:\Omega\to\R\,s.t.\, \int_\Omega |f(z)|^2\d z<\infty\}$. Finally, we denote by $\delta_{z_0}(\cdot)$ the Dirac delta at $z_0$.

\subsection{Strong and Weak Forms of the Dynamics}
Consider the discrete-time dynamics 
\begin{subequations}\label{root_eq:strong_dynamics}
    \begin{align}
    x^{t+1} &= A^* x^t + B^* u^t + w_x^t \,, \label{root_eq:strong_dynamics:a}\\
    y^t &= C^* x^t + w_y^t \,.    
\end{align}
\end{subequations}
with operators
\begin{align*}
        \text{Dynamics } \quad 
        A^*:\Ds(A^*) \to \X\,, \quad &\Ds(A^*)\subseteq \X \,,\\
        \text{Control } \quad
    B^*:\Ds(B^*) \to \X\,,\quad  &\Ds(B^*)\subseteq \U \,,\\
    \text{Observation } \quad
    C^*:\Ds(C^*) \to \Y\,, \quad &\Ds(C^*)\subseteq \X \,.
\end{align*}
The Hilbert spaces $\X,\U,\Y$ may be finite- or infinite-dimensional\footnote{An alternative is to use general dual spaces, going beyond Hilbert spaces.}. 
In infinite dimensions, a large class of \glspl{PDE} and integral equations can be expressed as \eqref{root_eq:strong_dynamics} for suitable choices of operators $A^*, B^*$ and $C^*$.
If $\{x^t\}_{t=0}^T=\{x^t(\cdot)\}_{t=0}^T$ and $\{w_x^t\}_{t=0}^T=\{w_x^t(\cdot)\}_{t=0}^T$ are functions over some space $\Omega$, Equation \eqref{root_eq:strong_dynamics:a} is said to hold \emph{strongly}, i.e. $x^{t+1}(z) = (A^* x^t)(z) + (B^* u^t)(z) + w_x^t(z) $ pointwise for all $z\in\Omega$. However, this notion of solution to \eqref{root_eq:strong_dynamics} may be too restrictive in many application settings. For instance, one may want to consider cases where \eqref{root_eq:strong_dynamics:a} is the time-distretization of an evolution equation 
\begin{equation}\label{eq:PDE-cont-t}
    \partial_tx(t,z) = L^*x(t,z)+B^*u(t,z)+w_x(t,z)\,,
\end{equation}
for some differential operator $L^*:\Ds(L^*)\subseteq\X\to\X$ (for Forward Euler discretization with time step $\Delta t$, one would recover \eqref{root_eq:strong_dynamics:a} by setting $A^*=\Id+\Delta t L^*$). If $(L^*x)(z):=v\cdot\nabla_z x(z)$ is the advection operator for some fixed velocity field $v\in\R^d$ and $B^*=0$, $w_x=0$, then a step function at the origin, $x^0(z):=\text{sign}(z)+1$, is propagated through the dynamics and remains a step function at all times. This solution, which is physically relevant, does not solve \eqref{eq:PDE-cont-t} strongly since a step function is not differentiable in the classical sense. Or, consider for instance dynamics $L^*$ such that the solution $x(t,z)$ develops shocks or discontinuities over time, loosing differentiability. To give a sense to 'solving \eqref{root_eq:strong_dynamics}' in these cases, we use a more general notion of solution: $(x^t,u^t)_{t=0}^T$ solves \eqref{root_eq:strong_dynamics} \emph{weakly} if
for all $f \in \Ds(A) \cap \Ds(B)\subseteq \X$ and $g \in \Ds(C) \subseteq \Y$,
\begin{subequations}\label{root_eq:weak_dynamics}
    \begin{align}
    \<x^{t+1},f>_\X &= \<x^t,A f>_\X + \<u^t, B f>_\U + \<w^t_x, f>_\X  \,,\label{root_eq:weak_dynamics_state} \\
    \<y^t,g>_\Y &= \<x^t,C g>_\X  + \<w^t_y, g>_\Y. \label{root_eq:weak_dynamics_output}
\end{align}
\end{subequations}
Here, $f,g$ are considered test functions, and applying differential operators to test functions instead of the solution itself allows for less regular solutions, see Example~\ref{example:H1}.
We will use the weak notion of solution throughout this paper, and all pointwise statements are thus to be understood in the almost everywhere sense.
\begin{example}
\label{example:finite-dim}
In finite dimensions, the weak and strong forms are equivalent, so our framework reduces to the standard \gls{SLS} setting. When $\X=\R^{n_x}$, $\U=\R^{n_u}$, and $\Y=\R^{n_y}$, the weak form of the dynamics \eqref{root_eq:weak_dynamics} reduces to the standard (strong form) finite-dimensional control setting, since $\X$, $\U$ and $\Y$ are Hilbert spaces when equipped with the $2$-norm. Observe that the domains of $A^*\in\R^{n_x \times n_x}$, $B^* \in \R^{n_x \times n_u}$, and $C^*\in\R^{n_y \times n_x }$ are the entire spaces $\X$ and $\U$, which is not necessarily true in the infinite dimensional setting. 
\end{example}

\begin{example}\label{example:H1}
    Consider $\X=\U=L^2(\R)$, and the dynamics
        \begin{align*}
        x^{t+1}(z) = \partial_z x^t (z) + \int b(z-z') u^t(z') \d z' \,,
    \end{align*}
    for $z\in\R$, $b\in L^2(\R)$ with $b(-z)=b(z)$.
    This corresponds to the strong form of the dynamics \eqref{root_eq:strong_dynamics:a} with $A^*=\partial_z$ the differentiation operator and $B^*$ the convolution operator. Here, 
$\Ds(A^*)=H^1(\R):=\{f\in L^2(\R)\,s.t. \int_\R|\partial_zf(z)|^2\d z<\infty\}$ and $B^*u^t = b\ast u^t$ with $\Ds(B^*)=\U$. The weak form \eqref{root_eq:weak_dynamics_state} of the dynamics is
    \begin{align*}
        \<x^{t+1},f>_{\X} &= -\<x^t,\partial_z f>_{\X} + \<u^t, b \ast f>_\U \,,
    \end{align*}
    for all $f\in\Ds(A)=H^1(\R)$, where $A=-\partial_z$ and $Bf=b\ast f$. Explicitly writing the $L^2$ inner products, the above means
    \begin{align*}
        & \int x^{t+1}(z) f(z) \d z = -\int x^t(z)\partial_z f(z) \d z   + \iint u^t(z) b(z-z') f(z') \d z \d z'\,, \quad \forall\ f \in H^1(\R)\,.
    \end{align*}
    The strong form requires $x^t\in H^1(\R)$ to be differentiable, whereas for the weak form $x^t\in L^2(\R)$ is sufficient, shifting regularity requirements to the test functions $f\in H^1(\R)$ instead. In this way the weak form gives sense to a non-differentiable function $x^t$ solving a differential equation. 
\end{example}
\subsection{System Level Synthesis}
Finite-dimensional \gls{SLS} can be used to optimize a linear controller while convexly enforcing locality, communication delay and speed constraints \cite{anderson2019system},
which is not possible when optimizing over 
the (linear) controller gains
directly.
In particular, locality constraints allow for large-scale problems to be solved in parallel, with controllers using only local information, leading to 
a scalable controller implementation. 
We leverage \gls{SLS} to generalize the convex controller synthesis to infinite dimensional settings, including deterministic and stochastic \gls{PDE} dynamics and integral equations.

In the weak sense the control input $u^t$ is given via a family of operators $K_x^{t,\tau}:\Ds(K_x^{t,\tau})\subseteq \U\to\X$ and $K_y^{t,\tau}:\Ds(K_y^{t,\tau})\subseteq \U\to\Y$
for $t,\tau\in\{0,\dots,T\}$ such that for all test functions $h\in\cap_{\tau=0}^t \Ds(K_{x,y}^{t,t-\tau})\subseteq \U$,
\begin{align*}
    \<u^t,h>_\U &= \sum_{\tau=0}^t\<x^\tau,K_x^{t,t-\tau} h>_\X \quad \text{for state feedback,} \\
    \<u^t,h>_\U &= \sum_{\tau=0}^t\<y^\tau,K_y^{t,t-\tau} h>_\Y \quad \text{for output feedback}\,.
\end{align*}
We will present results for both state feedback (SF) and output feedback (OF) settings. The controller is $u^t=\sum_{\tau=0}^t (K_x^{t,t-\tau})^* x^\tau$ in the SF setting, and $u^t=\sum_{\tau=0}(K_y^{t,t-\tau})^* y^{\tau}$ in the OF setting.

\begin{remark}
    Taking operator adjoints means switching the spaces for domains and ranges, so generalizing \gls{SLS} to operators in Hilbert spaces switches the conventional source and target spaces from existing \gls{SLS} theory. For example, if $B:\X \to \U$ then $B^*:\U \to \X$. For details, see \cite{showalter_hilbert_2009}. 
\end{remark}

\section{Results}
In this section, we show how to convexly synthesize a linear controller for the dynamics \eqref{root_eq:weak_dynamics}. 
Throughout, we consider a time horizon $T\in \mathbb{N}$ and denote
{\allowdisplaybreaks\begin{align*}
    x &\coloneqq \begin{bmatrix}
        x^0 \\  \vdots \\ x^T
    \end{bmatrix}\in \Xs, \  
    w_x \coloneqq \begin{bmatrix}
        x^0 \\ w^0  \\ \vdots \\ w^{T-1}
    \end{bmatrix}\in \Xs,  \ u = \begin{bmatrix}
        u^0 \\ \vdots \\ u^T
    \end{bmatrix}\in \Us,
    \end{align*}}and operator matrices $\A:\Ds(\A)\subseteq\Xs\to\Xs$, $\K_x:\Ds(\K_x)\subseteq\Us\to\Xs$ 
given by
    \allowdisplaybreaks\begin{align}\label{root_eq:A-K}
        \A  &\coloneqq \begin{bmatrix}
            0 & A & & \\
            & \ddots & \ddots & \\
            & & 0 & A\\
            & & & 0  
        \end{bmatrix}\,,
        \ \ \K_x \coloneqq \begin{bmatrix}
            K_x^{0,0} & \hdots & K_x^{T,T} \\  &  \ddots &  \vdots \\   &  & K_x^{T,0}
        \end{bmatrix}\,,
    \end{align} 
with $\B:\Ds(\B)\subseteq\Xs\to\Us$ defined similarly to $\A$, $\K_y:\Ds(\K_y)\subseteq \Us \to \Ys$ defined similarly to $\K_x$, and $\C:\Ds(\C)\subseteq \Ys\to\Xs$ given by $\C:=\blkdiag(C)$, that is, the matrix $C$ repeated on the block diagonals; all operator matrices are of size $(T+1) \times (T+1)$. Denote $\DD:=\Ds(\A)\cap\Ds(\B)= \X\times (\Ds(A)\cap\Ds(B))^{\otimes T}\subseteq \Xs$.
Using the full time horizon, the dynamics \eqref{root_eq:weak_dynamics} can be written, for all $f\in\DD$ and $g\in\Ds(\C)$, \begin{subequations}\label{root_eq:fullhorizon_dynamics}  \begin{align}
        \<x,f>_\Xs &= \<x,\A f>_\Xs + \<u,\B f>_\Us + \<w_x,f>_\Xs \,,
        \label{root_eq:fullhorizon_dynamics_state}  \\
        \<y,g>_\Ys &= \<x,\C g>_\Xs  + \<w_y,g>_\Ys \label{root_eq:fullhorizon_dynamics_output} \,.
    \end{align}
\end{subequations}    
The relationship between $u$ and $x$ is given by
\begin{subequations}\label{root_eq:controller}
    \begin{align}
    \<u,h>_\Us &= \<x,\K_x h>_\Xs  \quad \forall\, h\in \Ds(\K_x) \text{ for SF}, \label{root_eq:controller_state_fb} \\
    \<u,h>_\Us &= \<y, \K_y h>_\Ys \quad \forall\, h\in\Ds(\K_y) \text{ for OF.} \label{root_eq:controller_output_fb}
\end{align}
\end{subequations}
The matrix structure in \eqref{root_eq:A-K} arises by writing the strong form of the dynamics \eqref{root_eq:strong_dynamics} in vector notation, and then taking adjoints of the resulting operators with respect to $\<\cdot,\cdot>_{\Xs,\Us,\Ys}$. Note that the controller being causal means $\K^*_x$ and $\K^*_y$ are lower block-triangular, and hence $\K_x$ and $\K_y$ are upper block-triangular. 
    
\subsection{State Feedback}
We define closed-loop operators $\px:\Ds(\px)\subseteq\Xs \to \Xs$ and $\pu:\Ds(\pu)\subseteq\Us \to \Xs$ to map between the disturbance and the state and input respectively, in block-operator form
\begin{align}\label{root_eq:px}
    \px &\coloneqq \begin{bmatrix}
        \px^{0,0} &\px^{1,1} & \hdots & \px^{T,T}\\ 
         & \px^{1,0}  & & \px^{T,T-1} \\
          &  & \ddots & \vdots \\
         &  &  & \px^{T,0}
    \end{bmatrix} \,,
\end{align}  
where $\px^{t,\tau}:\Ds(\px^{t,\tau})\subseteq\X \to \X$ and $\pu^{t,\tau}:\Ds(\pu^{t,\tau})\subseteq\U \to \X$ are defined similarly. The adjoints $\px^*$, $\pu^*$ of the \glspl{CLM} $\px$ and $\pu$ are lower block-triangular to ensure causality, hence $\px$, $\pu$ are upper block-triangular, and they parameterize the trajectories $x \in \Xs$ and $u \in\Us$ via
\begin{subequations}\label{root_eq:CLMs_state_fb}
\begin{align}
    \<x,f>_\Xs &= \<w_x, \px f>_\Xs \quad \forall\, f\in\Ds(\px)\subseteq\Xs \,,\label{root_eq:CLMs_state_fb_x} \\
    \<u,h>_\Us &= \<w_x, \pu h>_\Xs \quad \forall\, h \in \Ds(\pu)\subseteq \Us\,. \label{root_eq:CLMs_state_fb_u}
\end{align}
\end{subequations}
In the following theorem, we state one of the main result of this paper: the relationship between the \glspl{CLM} $\px,\pu$ and the SF controller $\K_x$ is given by
\begin{align}\label{root_eq:K-rel_state_fb}
       \< f,\px^{-1} \pu h>_\Xs = \<f, \K_x h>_\Xs \,,
\end{align}
for appropriate test functions $f,h$. The \gls{SLP} in Theorem~\ref{root_thm:state_fb} parameterizes all possible causal $\px,\pu$.
\begin{theorem}[SLP-SF]\label{root_thm:state_fb}
    Fix disturbance function realization $w_x \in \Xs$ and operators $\A, \B$.
\begin{itemize}
    \item[(I)] If $\K_x$ of form \eqref{root_eq:A-K} is given such that $\Rs(\A),\Rs(\B)\subseteq\Ds(\K_x)$ and $\Rs(\K_x)\subseteq\DD$, then
    any trajectory $(x,u) \in \X\times\U$ satisfying the closed-loop dynamics \eqref{root_eq:fullhorizon_dynamics_state}-\eqref{root_eq:controller_state_fb} also satisfies \eqref{root_eq:CLMs_state_fb} with some causal \glspl{CLM} $\px,\pu$ satisfying $\Rs(\A)\subseteq \Ds(\px)$, $\Rs(\B)\subseteq \Ds(\pu)$ and 
\begin{equation}\label{root_eq:SLP_state_feedback} 
        \begin{split}
            \<f,\px \hat f>_\Xs = \<f,\px \A \hat f>_\Xs + \<f,\pu \B \hat f>_\Xs + \<f,\hat f>_\Xs\,, 
            \quad
             \forall\, f\in\Xs\,, \quad \forall\, \hat f\in\DD \,. 
        \end{split}
        \end{equation}

    \item[(II)]
    Let $\px:\Ds(\px)\to\Xs,\pu:\Ds(\pu)\to\Xs$ be arbitrary \glspl{CLM} satisfying \eqref{root_eq:SLP_state_feedback} such that $\px,\pu$ are upper block-triangular (implying that $\px$ is invertible), $\Rs(\A)\subseteq\Ds(\px)$, $\Rs(\B)\subseteq\Ds(\pu)$ and $\Rs(\pu)\subseteq\Ds(\px^{-1})$. Then the corresponding trajectory $(x,u)\in\X\times\U$ computed with \eqref{root_eq:CLMs_state_fb} 
    also satisfies \eqref{root_eq:fullhorizon_dynamics_state}, \eqref{root_eq:controller_state_fb} with the controller $\mathcal{K}_x$ defined by $\K_x:=\px^{-1}\pu$ and $\Ds(\K_x):=\Ds(\pu)$. 
\end{itemize}
\end{theorem}

\subsection{Output Feedback}
In the output feedback setting, the controller does not have access to full state information, only the observation data $y$. The \gls{CLM} operator matrices 
\begin{align*}
    &\pxx:\Ds(\pxx)\subseteq\Xs \to \Xs\,,\qquad
    \pux:\Ds(\pux)\subseteq \Us \to \Xs\,,\\
    &\pxy:\Ds(\pxy)\subseteq\Xs \to \Ys\,,\qquad
    \puy:\Ds(\puy)\subseteq\Us \to \Ys\,,
\end{align*}
are defined as block-triangular operator matrices with the same structure \eqref{root_eq:px} as $\px$ and $\pu$ in the state feedback setting. As shown in the following theorem, the trajectory $(x,u)\in\Xs\times\Us$ is parameterized by
\begin{subequations}\label{root_eq:CLMs_output_fb}
    \begin{align}
        \<x,f>_\Xs &= \<w_x, \pxx f>_\Xs + \<w_y, \pxy f>_\Ys \,,\label{root_eq:CLMs_output_fb_x}\\
        \<u,h>_\Us &= \<w_x, \pux h>_\Xs + \<w_y, \puy h>_\Ys \,,\label{root_eq:CLMs_output_fb_u} \\
        \forall\ f\, \in \Ds(\pxx)\cap\Ds(\pxy) &\subseteq\Xs\,, \qquad 
        \forall\ h\, \in\Ds(\pux)\cap \Ds(\puy)  \subseteq \Us\,. \notag
    \end{align}
\end{subequations}
We will show a relationship between the closed-loop maps and the controller $\K_y:\Ds(\K_y)\to\Y$ given by
\begin{align}\label{root_eq:K-rel_output_fb}
  \<g ,(\puy - \pxy \pxx^{-1} \pux )h >_\Ys = \< g,\K_y  h>_\Ys\,,
\end{align}
for appropriate test functions $g,h$.
\begin{theorem}[SLP-OF]\label{root_thm:output_fb}
 Fix disturbance function realizations $w_x \in \Xs$, $w_y\in\Ys$ and operators $\A, \B, \C$ with $\Rs(\A)\subseteq\DD$ and $\Rs(\C)=\DD$.
\begin{itemize}
    \item[(I)] If $\K_y$ of form \eqref{root_eq:A-K} is given such that $\Rs(\B)\subseteq\Ds(\K_y)$ and $\Rs(\K_y)\subseteq\Ds(\C)$, then
    any trajectory $(x,u) \in \X\times\U$ satisfying the closed-loop dynamics \eqref{root_eq:fullhorizon_dynamics_state}, \eqref{root_eq:fullhorizon_dynamics_output}, \eqref{root_eq:controller_output_fb} also satisfies \eqref{root_eq:CLMs_output_fb} with some \glspl{CLM} $(\pxx,\pxy,\pux,\puy)$ satisfying
    \begin{align*}
       &\Ds(\pxx)=\Ds(\pxy)=\DD\,,\quad
       \Ds(\pux)=\Ds(\puy)=\Ds(\K_y)\,,\\
       &\Rs(\pxx)=\Rs(\pux)=\DD\,,\quad
       \Rs(\pxy)=\Rs(\puy)= \Rs(\K_y)\,,
    \end{align*}
    and given by
\begin{subequations}\label{root_eq:SLP_output_feedback} 
             \begin{align}
                \<f,\pxx \hat f>_\Xs &=  \<f,\pxx \A \hat f>_\Xs + \<f,\pux \B \hat f>_\Xs + \<f,\hat f>_\Xs\,,\label{eq:SLP_OFa}\\
            \<g,\pxy \hat f>_\Ys &= \<g, \pxy \A \hat f>_\Ys + \<g,\puy \B \hat f>_\Ys\,,
            \label{eq:SLP_OFb}\\
            \<f,\pxx \hat f>_\Xs &= \<f, \A \pxx \hat f>_\Xs + \<f, \C \pxy \hat f >_\Xs + \<f, \hat f>_\Xs\,,
            \label{eq:SLP_OFc}\\
            \<f,\pux h>_\Xs &= \<f, \A \pux h>_\Xs + \<f, \C \puy h>_\Xs\,,\label{eq:SLP_OFd} \\
            &\forall f\in\Xs, g\in\Ys, \hat f\in \DD, h\in \Ds(\pux)\,. \notag
            \end{align}
        \end{subequations}
\item[(II)] Let $(\pxx,\pxy,\pux,\puy)$ be arbitrary \glspl{CLM} satisfying \eqref{root_eq:SLP_output_feedback} with causal upper block-triangular structure (implying that $\pxx$ is invertible),  with
    \begin{align*}
       &\Ds(\pxx)=\Ds(\pxy)=\DD\,,\quad
       \Rs(\B)\subseteq\Ds(\pux)=\Ds(\puy)\,,\\
       &\Rs(\pxx)=\Rs(\pux)=\DD\,,\quad
       \Rs(\pxy)=\Rs(\puy)=\Ds(\C)\,.
    \end{align*} 
    Then the corresponding trajectory $(x,u)\in\X\times\U$ computed with \eqref{root_eq:CLMs_output_fb} and output $y\in\Ys$ computed with \eqref{root_eq:controller_output_fb}
    also satisfy \eqref{root_eq:fullhorizon_dynamics},  with the controller $\K_y$ defined by 
    $$\K_y:=\puy - \pxy (\pxx^{-1}) \pux\,,$$
    and $\Ds(\K_y):=\Ds(\pux)=\Ds(\puy)$; it follows that $\Rs(\K_y)=\Ds(\C)$.
\end{itemize}
\end{theorem}

\subsection{Optimization and Structural Constraints}
\label{sec:constraints}
Classical control problems aim to construct controllers that minimize some cost, usually in terms of the state $x$ and input $u$. In the \gls{SLS} framework, it was shown in \cite[Section 2.2.2]{anderson2019system} that optimizing over a weighted Frobenius norm of (finite-dimensional) $\px$ and $\pu$ solves the linear quadratic regulator (LQR) problem. Other common objectives such as $\Hc_\infty$ and $\L_1$
can also be posed in terms of closed-loop maps (see, e.g., \cite{chen2024robust,anderson2019system,dean2020robust} for \gls{SLS}-based robustness guarantees). In the same vein, we propose optimizing over (infinite-dimensional) $\px$ and $\pu$ instead of the state and input. In the state feedback setting, the (infinite-dimensional) convex optimization problem given a cost $J$ is
\begin{align*}
    \min_{\px,\pu} J(\px,\pu) \quad \text{such that }\eqref{root_eq:SLP_state_feedback} \text{ holds,}
\end{align*}
and in the output feedback setting,
\begin{align*}
    \min_{\pxx,\pux,\puy,\pxy} J(\pxx,\pux,\puy,\pxy) \quad \text{such that }\eqref{root_eq:SLP_output_feedback} \text{ holds.}
\end{align*}
Let $\pCLMs$ be the set of closed-loop maps, with $\pCLMs=(\px,\pu)$ in the state feedback setting and $\pCLMs=(\pxx,\pux,\puy,\pxy)$ in the output feedback setting. 
A key feature of the original finite-dimensional \gls{SLS} formulation \cite{anderson2019system} is its ability to impose structural constraints such as communication delay, sensor delay and locality constraints by constraining the support of $\pCLMs$, while maintaining the convexity of the optimization problem.
This proposed infinite-dimensional formulation has the same properties. When restricting to a constraint set $S$, the optimization problem then becomes
\begin{align*}
    \min_{\theta \in S} J(\theta) \quad \text{s.t. } \eqref{root_eq:SLP_state_feedback} \text{ or } \eqref{root_eq:SLP_output_feedback} \text{ holds.}
\end{align*}

The following example shows how to construct such constraint sets $S$.
\begin{example}[Integral Dynamics with Constraints]\label{ex:integral_dynamics_constrained}
    Let $\Omega\subset \R^2$ be a compact set. 
    Consider the system with full state measurements
    $x^t\in L^2(\Omega)$, $u^t\in \R^{n_u}$, parameterized by $a\in L^2(\Omega)$ and $b:\Omega \to \R^{n_u}$,
\begin{align*}
    A^* x^t(\cdot) &= (a \ast x^t)(\cdot) \in L^2(\Omega)\,, \quad \forall \, x^t\in L^2(\Omega)\,,\\
    B^* u^t &= \sum_{l=1}^{n_u}  b(l,z) u^t_l\, \in L^2(\Omega)\,, \quad \forall \, u^t\in \R^{n_u}\,.
\end{align*}
This means \eqref{root_eq:weak_dynamics} is a simplified scattering equation, also known as linear Boltzmann, and is used to model phenomena like bacterial movement \cite[Sec 5.6]{perthame_transport_2007}. The functions $a$ and $b$ have compact support; for $a$, this specifies that $\A$ is \textit{local}, i.e., the state at $z$ affects only nearby states at the next time step. For $b$, the control affects states around the location of each actuator, where the location of the $i^{th}$ of $n_u$ actuators is located at $\hat z^{(i)}\in\R^2$. Assume $\supp a(\cdot) = \{ z\in\Omega ~|~ \norm{z}_2 \le r \}$ and $\supp b(l,\cdot) = \{ z \in \Omega \ | \ \norm{z-\hat z^{(l)}}_2 \le r \}$,
where $r$ is sufficiently small.
For this problem set-up, we optimize over causal operators $\pCLMs^*=(\px^*,\pu^*)$ that can be expressed via kernels: they are upper block-triangular as in \eqref{root_eq:px}, and for $f\in L^2(\Omega)$ and $z\in\Omega$,
\begin{align*}
   &((\px^{t,\tau})^*f)(z)=\int_\Omega \pkx^{t,\tau}(\tilde z,z) f(\tilde z)\d \tilde z\,,\\
   &((\pu^{t,\tau})^*f)_l=\int \pku^{t,\tau}(\tilde z, \hat z^{(l)}) f(\tilde z)\d\tilde z\,,
\end{align*}
with kernels $\pkx^{t,\tau}(\tilde z,z)\in L^2(\Omega\times\Omega)$, $\pku^{t,\tau}(\tilde z,z)\in L^2(\Omega,C(\Omega))$.
To impose locality constraints on $\px^*$ and $\pu^*$, we fix the constraint set $S$ to operators with restricted support of their corresponding kernels: $\supp \pkx^{t,\tau}=\supp \pku^{t,\tau} = \left\{\tilde z, z \in \Omega\ \big | \ \norm{\tilde z - z}_2 \le r \right\}$.
\end{example}
Note that in the general linear operator setting, the system is often not controllable due to under-actuation. This makes it more difficult to ensure that the dynamics and the support constraints are feasible, i.e., there exists some $\pCLMs^*$ that satisfies both. An interesting question for future work is rigorously defining notions of controllability on subspaces.

\subsection{Parallel Solving for Integral Transforms}\label{sec:Parallel}
Consider the state space $\X=L^2(\Omega)$ for some $\Omega\subset\R^{d}$, and consider \glspl{CLM} with the following integral transform structure: for $f\in\X$,
\begin{align*}
   &(\px^{t,\tau})^*f=\int_\Omega \pkx^{t,\tau}(\tilde z,\cdot) f(\tilde z)\d \tilde z\,,\\
   &(\pu^{t,\tau})^*f=T_u \circ \int_\Omega \pku^{t,\tau}(\tilde z,\cdot) f(\tilde z)\d\tilde z\,,
\end{align*}
where $T_u:\V\to\U$ for a suitable functional space $\V$ over $\Omega$ (not necessarily a Hilbert space), and kernels $\pkx^{t,\tau}\in L^2(\Omega\times\Omega)$, $\pku^{t,\tau}\in L^2(\Omega,\V)$.
We call $\M$ the set of all such admissible causal kernels $\pk=(\pkx^{t,\tau},\pku^{t,\tau})$, and define for fixed $\tilde z\in\Omega$ the set 
$\M_1:=\{\pk\,|\, \pkx^{t,\tau}(\tilde z,\cdot)\in\X, \pku^{t,\tau}(\tilde z,\cdot)\in\V \}$
and for fixed $z\in\Omega$ the set $\M_2:=\{\pk\,|\, \pkx^{t,\tau}(\cdot,z), \pku^{t,\tau}(\cdot,z)\in\X \}$.
Optimizing over \glspl{CLM} is equivalent to optimizing over kernels $\pk\in\M$.
When there exists a feasible set $S$ of locality constraints, it is possible to solve this 
optimization problem in parallel.

Consider cost functions $J_i:\M_i\to\R\cup\{+\infty\}$ and $J:\M\to\R\cup\{+\infty\}$ for $i=\{1,2\}$.
Denote the constraints by $P$, where $P$ collects the controller parametrization \eqref{root_eq:SLP_state_feedback} or \eqref{root_eq:SLP_output_feedback}, and the structural constraints $S$.
\begin{definition}[Disturbance-Parallelizable]\label{def:dist-para} 
$J$ and $P$ are \textit{disturbance-parallelizable} if there exists an increasing function $\varphi_1$ and norm $\|\cdot\|_1$ such that $J(\pk)=\varphi_1(\|\psi_1\|_1)$ for $\psi_1(\tilde z):=J_1(\pk(\tilde z, \cdot))$, and 
\begin{align*}
    \overline \pk = \argmin_{ \pk(\tilde z,\cdot) \in \M_1} J_1( \pk(\tilde z,\cdot)) \quad \text{s.t. } P(\pk(\tilde z, \cdot))=0 \quad \forall \tilde z\,,
\end{align*} 
results in the same minimizer as the full problem, that is,
\begin{align*}
    \overline{\pk} = \argmin_{ \pk \in \M } J( \pk) \quad \text{s.t. } P(\pk)=0\,.
\end{align*}
\end{definition}
\begin{definition}[$xuy$-Parallelizable]\label{def:xuy-para} 
$J$ and $P$ are \textit{$xuy$-parallelizable} 
if there exists an increasing function $\varphi_2$ and norm $\|\cdot\|_2$ such that $J(\pk)=\varphi_2(\|\psi_2\|_2)$ for $\psi_2(z):=J_2(\pk( z, \cdot))$, and 
\begin{align*}
    \overline \pk = \argmin_{  \pk(\cdot,z) \in \M_2}J_2( \pk(\cdot,z)) \quad \text{s.t. } P(\pk(\cdot, z))=0 \quad \forall z \,,
\end{align*} 
results in the same minimizer as the full problem, that is,
\begin{align*}
    \overline{\pk} = \argmin_{ \pk \in \M}J( \pk) \quad \text{s.t. } P(\pk)=0\,.
\end{align*}
\end{definition}
Disturbance-parallelizable is analogous to column-separable in finite dimensional \gls{SLS}; $xuy$-parallelizable is analogous to row-separable in finite dimensional \gls{SLS}.
\begin{remark}
If the system is parallelizable or separable into disturbance- and $xuy$-parallelizable components, 
results from \cite{bot_admm_2019}, and finite-dimensional \gls{SLS} \cite{anderson2019system, AmoAlonso_MPC} suggest Alternating Direction Method of Multipliers (ADMM) as a potential distributed solver implementation.
\end{remark}
We will use a disturbance-parallelizable objective and constraint in the following section and illustrate results with parallel computation.
\section{Control Implementation}
\label{sec:implementation}
In what follows, consider the dynamics from Example~\ref{ex:integral_dynamics_constrained},
\begin{equation}
\label{root_eq:dyn_implementation}
x^{t+1}(z) = \int a(z-z') x^t(z') d z' + \sum_{l=1}^{n_u} b(l,z) u_l^t\,.   
\end{equation}
Here, we can take $\V=C(\Omega)$ the continuous functions on $\Omega$, and $T:C(\Omega)\to\R^{n_u}$ to be the evaluation operator at the actuators, $(Tv)_l:=v(\hat z^{(l)})$.

From \eqref{root_eq:SLP_state_feedback}, the dynamics for the \gls{CLM} kernels $\pk$ are 
\begin{align*}
\pkx^{t+1,\tau+1}
(\tilde z, z) &= (A^* \pkx^{t,\tau}(\tilde z,\cdot))(z) + (B^*T_u\circ\pku^{t,\tau}(\tilde z,\cdot))(z) 
= \int a(z-z') \pkx^{t,\tau}(\tilde z,z') \d z' + \sum_{l=1}^{n_u} b(l,z) \pku^{t,\tau}
(\tilde z,\hat z^{(l)})\,,\\
\px^{t,0}(\tilde z, z) &= \delta_0(\tilde z- z)\,,
\end{align*}
where $\tilde z$ is the input, $z$ is the output. Define the cost function
\begin{align*}
J(\pk) := \sum_{t,\tau} \left(Q \iint |\pkx^{t,\tau}(\tilde z,z) |^2 \d \tilde z \d z + R \int \norm{T_u\circ\pku^{t,\tau}(\tilde z,\cdot)}^2_2 \d \tilde z \right),
\end{align*}
for scalars $Q>0$ and $R\ge 0$, analogous to LQR for finite-dimensional \gls{SLS}. Note that this means we seek a minimum operator norm solution for $(\px^{t,\tau})^*$ and $(\pu^{t,\tau})^*$, and with this choice the optimization problem and constraints are disturbance-parallelizable according to Definition~\ref{def:dist-para}, that is, we can optimize for each $\tilde z\in\Omega$ 
separately.

We use real Fourier basis functions for $\pkx$, $a$, and $b$
\footnote{For details about the choice of basis, and the CVXPY \cite{diamond2016cvxpy} quadratic program implementation of constraint \eqref{root_eq:SLP_output_feedback} for the dynamics \eqref{root_eq:dyn_implementation}, see \lec{Appendix~\ref{sec:implementation_details} and }\url{https://github.com/LEConger/SLS-evolution-equations}.}.
The basis allows us to solve the optimization problem over Fourier coefficients, rather than over function spaces, with truncation giving an approximate solution.
Additionally, we constrain the support of $\pkx$ and $\pku$ without affecting the convexity of the problem, as shown in Example~\ref{ex:integral_dynamics_constrained}. 
The constraint enforces locality on the disturbance propagation, and the controller. This locality constraint is enforced via basis functions with matching support.

\subsection{Simulation Results}
 The simulation results show selecting an expressive basis and a small number of approximation functions ($k=12$) are sufficient for the results to be accurate within $0.3\%$. The time horizon is $T=5$ with $n_u=16$ controllers over a domain $[-2,2]$ in both directions.
The controller topology and state response to a disturbance at position $(-0.26,  0.56)$ for the first two time steps are shown in Figure~\ref{fig:system_response}. The state is hit with a disturbance at time $t=0$, and the nearest controller responds at time $t=1$. By the second time step, the state is nearly back to zero. 
\begin{figure}[!ht]
    \centering
    \includegraphics[width=0.60\linewidth]{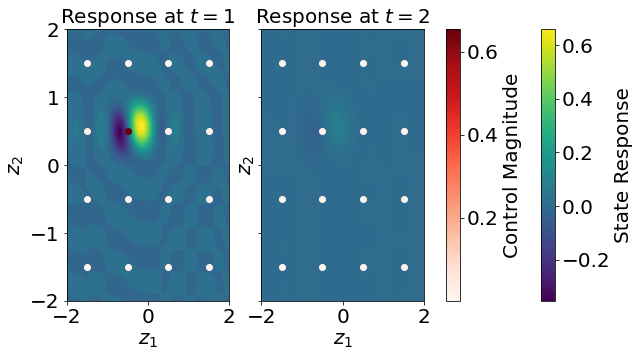}
    \caption{Optimal state response (continuous green shade) and control magnitude (red dot) after $t=1$, and $t=2$.}
    \label{fig:system_response}
\end{figure}

\textbf{Accuracy of basis functions:} To determine the accuracy of the basis functions, the state sequence given the controllers resulting from the basis function coefficient optimization was computed using numerical integration and compared with $\px$. The accuracy at each time step is summarized in Table~\ref{tab:error_and_performance}, with at most $0.23\%$ error from the exact solution.

\begin{table}[!ht]
    \caption{Relative error and performance gain at each time step.}
    \centering
    \begin{tabular}{c|c|c}
        time step & error (\%) & perf. gain (\%)\\
        \hline
        1 & 0.16 & 42.23 \\
        2 & 0.11 & 61.23\\
        3 & 0.17 & 70.93\\
        4 & 0.21 & 77.39\\
        5 & 0.23 & 82.16\\
    \end{tabular}
    \label{tab:error_and_performance}
\end{table}

\textbf{Comparison with uncontrolled system:}
To determine how much the controllers influence the state of the system, we compare the optimal state 
response
with its uncontrolled counterpart, i.e., $\theta_u =0$. The average 2-norm of the state at each time step is given in Table~\ref{tab:error_and_performance}. Compared with the state response with no control, the controller causes the state norm to decrease, e.g., by $42.23\%$ at $t=1$.

\textbf{Parallel control with constraints:} To illustrate how to compute controllers locally for large-scale systems, we increase the state domain and controllers, and we constrain the controllers to respond only to disturbances within a 1-norm radius of length 2. The \gls{SLS} optimization problem is solved in parallel for each of the 15 disturbances. For controllers responding to multiple disturbances, the linearity of the \glspl{CLM} allows us to add the maps together after synthesis. See Figure~\ref{fig:distributed_SLS} to observe the locality constraint held for each disturbance.
As is the case for discrete-space \gls{SLS} \cite{yu2021localized}, the disturbance-separability allows us to compute responses to multiple disturbances in parallel.

\begin{figure}[!ht]
    \centering \includegraphics[width=0.6\linewidth]{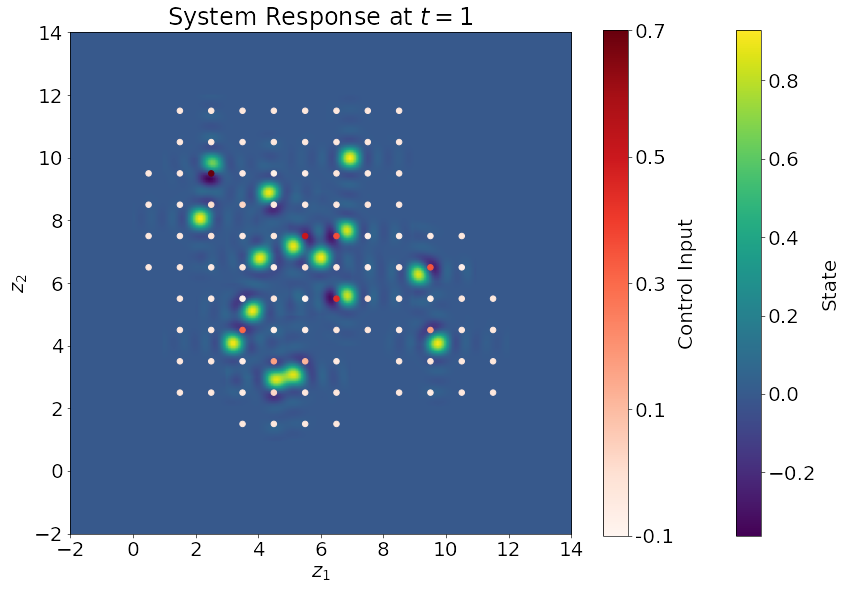} 
    \caption{The state and control responses to each of the 15 disturbances are computed in parallel, and then summed to obtain the final control values. The closed-loop maps are constrained to have support locally around each disturbance.}
    \label{fig:distributed_SLS}
\end{figure}

\textbf{Benefit of continuous space control:}
To highlight the benefits of our approach, we compare its performance against a classic discretize-then-optimize \gls{SLS} formulation. After spatially discretizing the dynamics \eqref{root_eq:dyn_implementation} for different discretization steps $dx$, the optimal discrete controller computed with finite-dimensional \gls{SLS} \cite{anderson2019system} was integrated exactly in the dynamics.
In Table \ref{tab:performance_discretization}, finite-dimensional \gls{SLS} and continuous space \gls{SLS} are compared at each time step against the exact response of a zeros input controller.

In particular, our approach leads to higher performance for this numerical example, compared to its finite-dimensional counterpart \cite{anderson2019system} for all considered discretizations, highlighting the benefits of our approach for continuous-space systems.
Importantly, tailored (finite-dimensional) \gls{SLS} solvers achieve a state dimension scalability of $\mathcal{O}(n_x^3)$ \cite{LEEMAN2024_fastSLS}, whereas our approach remains independent of the discretization used.
\begin{table}[ht!]
    \caption{Performance comparison with finite-dimensional \gls{SLS} \cite{anderson2019system}.}
    \centering
    \begin{tabular}{c|c|c}
        discretization step &avg  perf. gain (\%) & state dimension $n_x$ \\
        \hline
        continuous (our approach) & 42.79 & / \\
        $dx = 0.1$ & 32.26 &1600 \\
        $dx = 0.2$ & 30.54 &400\\
        $dx = 0.25$ & 31.91& 256 \\
        $dx = 0.5$ & 37.36 & 64\\
    \end{tabular}
    \label{tab:performance_discretization}
\end{table}
\section{Conclusion}
We proposed a new convex output-feedback controller synthesis framework for infinite-dimensional 
evolution equations, based on \gls{SLS}. The framework allows for convex structural constraints, such as sensor delay, communication delays, and locality constraints. Due to its similarities with finite-dimensional \gls{SLS}, we expect many \gls{SLS} properties to translate to this infinite-dimensional setting, e.g., efficient distributed solving, robustness guarantees, or robust constraint satisfaction.

\section{Acknowledgements}
The authors thank James Anderson, Samir Bennani, John C. Doyle and Dennice Gayme for valuable discussions.

\appendix
\section*{Appendix}
\section{Theorem Proofs}

\sloppy\begin{proof}[Proof of Theorem~\ref{root_thm:state_fb}]
    Proof of (I): Using the weak form of the dynamics \eqref{root_eq:fullhorizon_dynamics_state} and relation \eqref{root_eq:controller_state_fb}, $x$ satisfies 
    \begin{align*}
       & \<x,f>_\Xs = \<x, (\A + \K_x\B) f>_\Xs + \<w_x,f>_\Xs \quad \forall\, f\, \in \DD \\
        &\quad\Rightarrow \quad \<x,(\Id-\A-\K_x\B)f>_\Xs = \<w_x,f>_\Xs\quad \forall\, f\, \in \DD\,. 
    \end{align*}
    The operator matrix $\Id - \A - \K_x\B$ has domain and range  $\Ds(\Id - \A - \K_x\B)=\Rs(\Id - \A - \K_x\B)=\DD$, is upper block-triangular with identity operators on the diagonal, and therefore invertible with inverse given explicitly in terms of $A, B$ and $K_x^{t,\tau}$. Note that also $\Ds((\Id - \A - \K_x\B)^{-1})= \Rs((\Id - \A - \K_x\B)^{-1})=\bar \D$. Denote $\hat f:=(\Id - \A - \K_x\B)  f$.  Hence
    \begin{align*}
        \<x,\hat f>_\Xs = \<w_x,(\Id - \A - \K_x\B)^{-1}\hat f >_\Xs\,,
    \end{align*}
    Define $\px:=(\Id - \A - \K_x\B)^{-1}$ with $\Ds(\px):=\DD$ and $\pu:=\px \K_x$ with $\Ds(\pu)=\Ds(\K_x)$. By assumption, $\Rs(\A)\subseteq\Ds(\K_x)\subseteq\DD$, and so $\Rs(\A)\subseteq\Ds(\px)$. Trivially, $\Rs(\B)\subseteq \Ds(\pu)$. Finally, $\px$ and $\pu$ satisfy \eqref{root_eq:SLP_state_feedback}: for all $f\in\Xs, \hat f\in\DD$,
    \begin{align*}
        &\<f,(\Id - \A - \K_x\B)^{-1} \hat f>_\Xs = \<f,(\Id - \A - \K_x\B)^{-1}\A \hat f>_\Xs  + \< f,(\Id - \A - \K_x\B)^{-1} \K_x \B \hat f>_\Xs + \<f,\hat f >_\Xs \,, \\ 
        &\quad \Leftarrow \quad \< f,(\Id - \A - \K_x\B)^{-1}(\Id - \A - \K_x\B)\hat f >_\Xs = \<f,\hat f>_\Xs\,.
    \end{align*}
    Proof of (II): Since $\px,\pu$ are upper-block-diagonal, and satisfy \eqref{root_eq:SLP_state_feedback}, it follows that $\px^{t,0}=\Id$ for all $t\in\{0,\dots,T\}$; hence $\px$ is invertible. Define $\K_x:=\px^{-1}\pu$ and $\Ds(\K_x):=\Ds(\pu)$. We begin by showing \eqref{root_eq:controller_state_fb}. Indeed, using \eqref{root_eq:CLMs_state_fb}, for any $h\in\Ds(\pu)=\Ds(\K_x)$, we have
    \begin{align*}
        \<u,h>_\Us = \< w_x, \pu h>_\Xs=\< w_x, \px\K_x h>_\Xs
        =  \<x,\K_x h >_\Xs \,.
    \end{align*}
To show \eqref{root_eq:fullhorizon_dynamics_state}, recall the weak form of the \gls{SLP},
    \begin{align*}
        \<f,\px  \hat f >_\Xs = \< f, \px \A \hat f>_\Xs + \< f, \pu \B \hat f>_\Xs + \<f,\hat f>_\Xs \,,
    \end{align*}
    for all $f\in\Xs, \hat f\in\DD$,
    which can be rewritten as
    \begin{align*}
        &\<f,\px (\Id-\A) \hat f>_\Xs - \<f,\pu\B \hat f >_\Xs = \<f,\hat f>_\Xs \\
        &\quad \Rightarrow \quad  \<f,\px \tilde f>_\Xs  = \<f,(\Id-\A-\K_x\B)^{-1}\tilde  f>_\Xs\,,
    \end{align*}
    where we defined $\tilde f :=(\Id-\A-\K_x\B) \hat f$. In fact, thanks to the block-triangular structure of the operators, $\Id-\A-\K_x\B$ is invertible and therefore for any $\tilde f\in\Xs$ there exists $\hat f\in\DD$ such that this relationship holds. 
     Then selecting $f=w_x\in\Xs$, we have for any $x$ solving \eqref{root_eq:CLMs_state_fb_x}, 
    \begin{align*}
        &\<x,\tilde f>_\Xs = \< w_x,\px(\Id - \A - \K_x\B) \hat f>_\Xs  = \<w_x,\hat f>_\Xs\,.
    \end{align*}
    Substituting for $\tilde f$ on the left-hand side and rearranging,
    \begin{align*}
       \<x, \hat f>_\Xs &= \<x,\A\hat f>_\Xs + \<x,\K_x\B\hat f>_\Xs +  \<w_x,\hat f>_\Xs\\
       &= \<x,\A\hat f>_\Xs + \<u,\B\hat f>_\Us +  \<w_x,\hat f>_\Xs + \left(\<x,\K_x\B\hat f>_\Xs - \<u,\B\hat f>_\Us\right)\,.
    \end{align*}
    The last term in brackets vanishes by \eqref{root_eq:controller_state_fb},
   and so the desired system response is achieved. 
\end{proof}
\begin{proof}[Proof of Theorem~\ref{root_thm:output_fb}]
   Proof of (I): Given a causal $\K_y$ and using \eqref{root_eq:controller_output_fb}, the dynamics \eqref{root_eq:fullhorizon_dynamics} can be rearranged: for any $f\in\DD$, 
    \begin{align*}
        &\<x,(\Id-\A)f>_\Xs = \<y, \K_y \B f >_\Ys + \<w_x,f>_\Xs  =\< x, \C \K_y \B f >_\Xs +  \<w_y,\K_y \B f>_\Ys + \<w_x,f>_\Xs  \\
        & \quad \Rightarrow \quad   \<x,(\Id-\A-\C\K_y\B ) f>_\Xs = \<\wy,\K_y\B f >_\Ys + \<w_x,f>_\Xs \,.
    \end{align*}
    The operator matrix $\tilde\A \coloneqq \Id-\A-\C\K_y\B$ has domain $\Ds(\tilde A)=\DD$, range $\Rs(\tilde A)=\DD$, and is invertible because it is upper block-triangular with identity operators on the block diagonal. Define $\hat f:=(\Id-\A-\C\K_y\B )f$. Then
    \begin{align*}
        \<x,\hat f>_\Xs = \<w_y,\K_y\B \tilde\A^{-1}\hat f >_\Ys+ \<w_x,\tilde\A^{-1} \hat f>_\Xs \,.
    \end{align*}
    By \eqref{root_eq:controller_output_fb} and since $\Rs(\K_y)\subseteq\Ds(\C)$, for all $h\in\D(\K_y)$,
    \begin{align*}
        \<u,h>_\Us = \<y,\K_y h>_\Ys = \<x,\C \K_y h>_\Xs  + \<w_y,\K_y h>_\Ys\,.
    \end{align*}
    Define the closed-loop maps 
    \begin{gather*}
        \pxx:=\tilde\A^{-1}\,,\qquad
        \pxy := \K_y\B\tilde\A^{-1}\,, \qquad
        \pux := \tilde\A^{-1} \C \K_y \,, \qquad
        \puy := \K_y\B\tilde\A^{-1} \C \K_y  + \K_y \,.
    \end{gather*}  
    Note that all conditions on domains and ranges for these operators are satisfied thanks to the assumptions. Next, we will show that the closed-loop maps satisfy the \gls{SLP} \eqref{root_eq:SLP_output_feedback}. We will check each of the four equalities. \eqref{eq:SLP_OFa} holds if for all $f\in\Xs,\hat f\in \DD$,
    \begin{align*}
        \<f,\tilde\A^{-1} \hat f>_\Xs = \<f,\tilde\A^{-1} \A \hat f>_\Xs
        + \<f, \tilde\A^{-1} \C \K_y \B \hat f>_\Xs + \<f,\hat f>_\Xs\,,
    \end{align*}
    which is equivalent to
       $ \<f,\tilde\A^{-1}\tilde\A \hat f>_\Xs = \<f,\hat f>_\Xs$.
   Eqn~\eqref{eq:SLP_OFb} holds if for all $g\in\Y, \hat f\in\DD$,
    \begin{align*}
        \<g,\K_y\B\tilde\A^{-1} \hat f>_\Ys = \<g, \K_y\B\tilde\A^{-1} \A \hat f>_\Ys 
        + \<g,(\K_y\B\tilde\A^{-1} \C+ \Id) \K_y \B \hat f>_\Ys  \,,
    \end{align*}
    which reduces to $\<g,\K_y\B \tilde\A^{-1}\tilde\A \hat f>_\Ys = \<g,\K_y \B \hat f>_\Ys$.
    \eqref{eq:SLP_OFc} holds if for all $f\in\Xs, \hat f\in \DD$,
    \begin{align*}
        \<f,\tilde\A^{-1} \hat f>_\Xs = \<f,\A \tilde\A^{-1} \hat f>_\Xs 
        + \<f, \C \K_y\B\tilde\A^{-1} \hat f>_\Xs + \<f,\hat f>_\Xs \,,
    \end{align*}
    which reduces to $ \<f,\tilde\A\tilde\A^{-1} \hat f>_\Xs = \<f, \hat f>_\Xs$.
\eqref{eq:SLP_OFd} holds if for all $f\in\Xs$ and $h\in\Ds(\K_y)$,
    \begin{align*}
        \<f,\tilde\A^{-1} \C \K_y h >_\Xs = \<f,\A \tilde\A^{-1} \C \K_y h>_\Xs 
        + \<f, \C (\K_y\B\tilde\A^{-1} \C \K_y + \K_y) h>_\Xs \,,
    \end{align*}
    which reduces to $ \<f,\tilde\A\tilde\A^{-1} \C \K_y h>_\Xs = \<f,\C \K_y h>_\Xs$.\\
    
    Proof of (II): Consider $\theta=(\pxx,\pxy,\pux,\puy)$ 
    solving the \gls{SLP} \eqref{root_eq:SLP_output_feedback}. Since all $\theta$ operators
    have the structure \eqref{root_eq:px}, it follows that $\pxx^{t,0}=\Id$ for all $t\in\{0,\dots,T\}$; hence $\pxx$ is invertible. 
    Defining the controller $\K_y=\puy - \pxy (\pxx^{-1}) \pux$ with $\Ds(\K_y)=\Ds(\pux)=\Ds(\puy)$,  we will show that $(x,u)$ defined via \eqref{root_eq:CLMs_output_fb} and output feedback $y$ defined via \eqref{root_eq:controller_output_fb} solve \eqref{root_eq:fullhorizon_dynamics}. First, note that $\Rs(\C)=\Rs(\pux$), $\pxx^{-1}$ is bijective on $\DD$ and $\Rs(\pxy)=\Rs(\puy)=\Ds(\C)$; hence $\Ds(\C)=\Rs(\K_y)$.
    From the \gls{SLP} \eqref{root_eq:SLP_output_feedback}, 
    \begin{align}
         \pux \B &= \pxx (\Id-\A)  - \Id \,,  \label{root_eq:pux_B}\\
         \puy \B &= \pxy (\Id-\A)\,, \label{root_eq:puy_B}\\
         \C \pxy  &= (\Id-\A)\pxx - \Id \,, \label{root_eq:C_pxy}
    \end{align}
    weakly hold on $\DD$, and 
    \begin{align}
        \C\puy=(\Id-\A)\pux \label{root_eq:C_puy}\,,
    \end{align}
    weakly on $\Ds(\pux)=\Ds(\puy)$.
To show \eqref{root_eq:fullhorizon_dynamics_state}, we use  \eqref{root_eq:CLMs_output_fb_x} to compute for all $f\in\DD$,
\begin{align*}
 &\<x,(\Id-\A)f>_\Xs  - \<u,\B f>_\Us - \<w_x,f>_\Xs=  \<w_x, \pxx (\Id-\A)f>_\Xs + \<w_y, \pxy (\Id-\A) f>_\Ys  - \<u,\B f>_\Us - \<w_x,f>_\Xs\,.
\end{align*}
The right-hand side vanishes, since by \eqref{root_eq:CLMs_output_fb_u} together with \eqref{root_eq:pux_B} and \eqref{root_eq:puy_B},
\begin{align*}
    &\<u,\B f>_\Us=\<w_x, \pux \B f>_\Xs + \<w_y, \puy \B f>_\Ys =\<w_x, \pxx(\Id-\A) f>_\Xs - \<w_x,f>_\Xs+ \<w_y, \pxy(\Id-\A) f>_\Ys \,,
\end{align*}
using $\Rs(\A)\subseteq\DD$ and $\Rs(\B)\subseteq\Ds(\pux)=\Ds(\puy)$.

To show \eqref{root_eq:fullhorizon_dynamics_output}:
Since $\Ds(\C)=\Rs(\K_y)$, for any $g\in\Ds(\C)$ there exists $h\in\Ds(\K_y)$ such that $g=\K_y h$. Hence from \eqref{root_eq:controller_output_fb} we have
\begin{align*}
    &\<y,g>_\Ys-\<x,\C g>_\Xs-\<w_y,g>_\Ys= \<u,h>_\Us -\<x,\C \K_y h>_\Xs-\<w_y,\K_y h>_\Ys\,.
\end{align*}
We will show that the right-hand side vanishes. Note that $\C \K_y h=\C g\in\DD$ since $\Rs(\C)\subseteq\DD$.
Then from \eqref{root_eq:CLMs_output_fb_x} together with \eqref{root_eq:C_pxy}-\eqref{root_eq:C_puy}, it follows that
\begin{align*}
 &\<x,\C \K_y h>_\Xs   
 = \<w_x, \pxx \C \K_y h>_\Xs + \<w_y, \pxy \C \K_y h>_\Ys\\
 &\ = \<w_x, \pxx \C \puy h>_\Xs
 -\<w_x, \pxx \C \pxy\pxx^{-1}\pux h>_\Xs + \<w_y, \pxy \C \puy h>_\Ys
 - \<w_y, \pxy \C \pxy\pxx^{-1}\pux h>_\Ys\\
 &\ = \<w_x, \pxx (\Id-\A)\pux h>_\Xs
 -\<w_x, \pxx  (\Id-\A)\pxx\pxx^{-1}\pux h>_\Xs+\<w_x,\pux h>_\Xs  - \<w_y, \pxy (\Id-\A) \pxx\pxx^{-1}\pux h>_\Ys\\
 &\quad + \<w_y, \pxy (\Id-\A)\pux h>_\Ys
+\<w_y,\pxy\pxx^{-1}\pux h>_\Ys\\
  &\ =\<w_x,\pux h>_\Xs+\<w_y,\pxy\pxx^{-1}\pux h>_\Ys\,.
\end{align*}
Using \eqref{root_eq:CLMs_output_fb_u}, we have 
\begin{align*}
&\<u,h>_\Us -\<x,\C \K_y h>_\Xs-\<w_y,\K_y h>_\Ys
\\
&\quad =
\<w_x,\pux h>_\Xs+\<w_y,\puy h>_\Ys
- \<w_x,\pux h>_\Xs-\<w_y,\pxy\pxx^{-1}\pux h>_\Ys -\<w_y,\K_y h>_\Ys\\
&\quad=\<w_y,\K_y h>_\Ys-\<w_y,\K_y h>_\Ys=0\,.
\end{align*} \end{proof}

\section{Implementation Details}\label{sec:implementation_details}
In this section, we provide implementation details for our numeric example, particularly regarding the basis function selection and optimization setup. Recall the dynamics from Example~\ref{ex:integral_dynamics_constrained},
\begin{equation*}
x^{t+1}(z) = \int a(z-z') x^t(z') d z' + \sum_{l=1}^{n_u} b(l,z) u_l^t\,.   
\end{equation*}
From \eqref{root_eq:SLP_state_feedback}, the dynamics for the \gls{CLM} kernels $\pk$ are 
\begin{align*}
\pkx^{t+1,\tau+1}
(\tilde z, z) &= (A^* \pkx^{t,\tau}(\tilde z,\cdot))(z) + (B^*T_u\circ\pku^{t,\tau}(\tilde z,\cdot))(z) 
= \int a(z-z') \pkx^{t,\tau}(\tilde z,z') \d z' + \sum_{l=1}^{n_u} b(l,z) \pku^{t,\tau}
(\tilde z,\hat z^{(l)})\,,\\
\px^{t,0}(\tilde z, z) &= \delta_0(\tilde z- z)\,,
\end{align*}
where $\tilde z$ is the input, $z$ is the output. Recall also the cost function
\begin{align*}
J(\pk) := \sum_{t,\tau} \left(Q \iint |\pkx^{t,\tau}(\tilde z,z) |^2 \d \tilde z \d z + R \int \norm{T_u\circ\pku^{t,\tau}(\tilde z,\cdot)}^2_2 \d \tilde z \right),
\end{align*}
for scalars $Q>0$ and $R\ge 0$, analogous to LQR for finite-dimensional \gls{SLS}.
We use real Fourier basis functions for $\vartheta_x$ and $\vartheta_u$. For $z=[z_1,z_2]$,
\begin{align*}
    \varphi_{mn1}(z) &= \sin\frac{2\pi m z_1}{\lambda_1} \sin \frac{2\pi n z_2}{\lambda_2}\,, \qquad \qquad
    \varphi_{mn2}(z) =  \cos\frac{2\pi m z_1}{\lambda_1} \sin \frac{2\pi  n z_2}{\lambda_2}\,, \\
     \varphi_{mn3}(z) &=  \sin\frac{2\pi m z_1}{\lambda_1} \cos \frac{2\pi n z_2}{\lambda_2}\,, \qquad \qquad
      \varphi_{mn4}(z) =  \cos\frac{2\pi m z_1}{\lambda_1} \cos \frac{2\pi n z_2}{\lambda_2} \,.
\end{align*}
The approximations for $a$, $b$, and $\vartheta_x$ are
\begin{align*}
    a(z) &\approx \sum_{m,n=0}^k \sum_{i=1}^4 a_{mni} \varphi_{mni}(z) \,, \qquad 
    b(z) \approx \sum_{m,n=0}^k \sum_{i=1}^4 \begin{bmatrix}
         b_{mni l} \varphi_{mni}(z) \\ \vdots \\ b_{mni n_u} \varphi_{mni}(z)
    \end{bmatrix} \in \R^{n_u}\,, \\
    \vartheta_x^{t}(\tilde z,z) &\approx \sum_{m,n=0}^k \sum_{i=1}^4 \alpha_{mni}(\tilde z) \varphi_{mni}(z)\,,
\end{align*}
for coefficients $a_{mni},b_{mnil},\alpha_{mni}$ defined as
\begin{align*}
    a_{mni} &=\frac{\kappa}{\lambda_1 \lambda_2} \int_{-\lambda_2/2}^{\lambda_2/2} \int_{-\lambda_1/2}^{\lambda_1/2} a(z) \varphi_{mni}(z) \d z_1 \d z_2\,, \qquad
    b_{mnil} =\frac{\kappa}{\lambda_1 \lambda_2} \int_{-\lambda_2/2}^{\lambda_2/2} \int_{-\lambda_1/2}^{\lambda_1/2} b(l,z) \varphi_{mni}(z) \d z_1 \d z_2\,,\\
    \kappa &= \begin{cases}
        1 & \text{if } m=0 \text{ and } n=0   \\
        2 & \text{if } m=0 \text{ or } n = 0 \\
        4 & \text{otherwise.}
    \end{cases}
\end{align*}

Note that this problem is input-separable (analogous to column-separable in finite dimensions) for a Frobenius norm cost function, so we can fix $\tilde z$ and solve for the operator in terms of $z$ only.
For implementation purposes, $\lambda_1$ and $\lambda_2$ must be at least twice the size of the domain in the $x$ and $y$ directions, respectively, to prevent aliasing.
The function $a(z)$ is
\begin{align*}
    a(z) = \begin{cases}
    \frac{1}{2} + \frac{1}{2}\cos \frac{\pi \lVert z \rVert_2 }{r} & \text{ if } \lVert z \rVert_2 \le r \\
    0 & \text{otherwise.}
    \end{cases}
\end{align*}
We numerically integrate to compute the coefficients $a_{mni}$. Because this takes some time, we save the coefficients so that we can load them instead of computing them in the future.
Next we compute the coefficients for $b(z)$, which we select to be equal to
\begin{align*}
b(z) = \begin{bmatrix}
 -a(z-u_{loc}^{(1)}) \\
  \vdots \\
  -a(z-u_{loc}^{(n_u)})
\end{bmatrix}\,,
\end{align*}
where $u_{loc}^{(l)}$ is the location of the $l^{th}$ controller.
Since this function has $n_u$ components, it is slower to compute the coefficients than for $a(z)$, so we suggest running the file 
\texttt{compute\_b\_coefficients\_parallel.py} which parallelizes the computation of the coefficients.
We use the system-level parameterization (SLP) to relate all of the coefficients and controller values. At time $t=0$, the first state response is 
\begin{align*}
\vartheta_x^{0}(\tilde z,z) = \delta(z-\tilde z)\,.
\end{align*}
The state response at time $t=1$, as specified by the SLP, is
\begin{align*}
\vartheta_x^{1}(\tilde z,z) &= \iint a(z-z') \vartheta_x^0(\tilde z,z') \d z + b(z)^\top \vartheta_u(\tilde z) = a(z-\tilde z) + b(z)^\top \vartheta_u(\tilde z)\,.
\end{align*}
Since $\vartheta_x^0(\tilde z,z)=\delta(z-\tilde z)$ is not a free variable, we start by parameterizing $\vartheta_z^1(\tilde z,z)$ in terms of $\vartheta_u^0(\tilde z)$. We will use the basis function approximation. We use the trigonometry identities
\begin{align*}
    \sin\frac{2\pi m}{\lambda_i}(z_i-\tilde{z}_i) &= \sin\frac{2\pi m}{\lambda_i}z_i \cos\frac{2\pi m}{\lambda_i}\tilde{z}_i-\cos\frac{2\pi m}{\lambda_i}z_i \sin\frac{2\pi m}{\lambda_i}\tilde{z}_i \\
    \cos\frac{2\pi m}{\lambda_i}(z_i-\tilde{z}_i) &= \cos\frac{2\pi m}{\lambda_i}z_i \cos\frac{2\pi m}{\lambda_i}\tilde{z}_i+\cos\frac{2\pi m}{\lambda_i}z_i \cos\frac{2\pi m}{\lambda_i}\tilde{z}_i
\end{align*}
to express $a(z-\tilde z)$ in terms of basis functions
\begin{align*}
    a(z-\tilde z) &\approx \sum_{m,n=0}^k \sum_{i=1}^4 a_{mni}\varphi_{mni}(z-\tilde z) =\sum_{m,n=0}^k \sum_{i=1}^k \hat a_{mni}(\tilde z) \varphi_{mni}(z)\,.
\end{align*}
The coefficients $\hat a_{mni}$ are
\begin{align*}
    \hat a_{mni} = \begin{cases}
        a_{mn1} \varphi_4(\tilde z) + a_{mn2} \varphi_3(\tilde z) + a_{mn3}\varphi_2(\tilde z) + a_{mn4}\varphi_1(z) & \text{if } i=1 \\
        -a_{mn2} \varphi_3(\tilde z) + a_{mn2} \varphi_4(\tilde z) - a_{mn3}\varphi_1(\tilde z) + a_{mn4}\varphi_2(z) & \text{if } i=2 \\
        -a_{mn3} \varphi_2(\tilde z) - a_{mn2} \varphi_1(\tilde z) + a_{mn3}\varphi_4(\tilde z) + a_{mn4}\varphi_3(z) & \text{if } i=3 \\
        a_{mn4} \varphi_1(\tilde z) - a_{mn2} \varphi_2(\tilde z) - a_{mn3}\varphi_3(\tilde z) + a_{mn4}\varphi_4(z) & \text{if } i=4 \\
    \end{cases}\,.
\end{align*}
Therefore $\vartheta_x^{(1)}(\tilde z)$ is
\begin{align*}
\vartheta_x^1(\tilde z) = \sum_{m,n=0}^k \sum_{i=1}^4 \alpha_{mni}^{(1)}(\tilde z) \varphi_{mni}(z) = \sum_{m,n=0}^k \sum_{i=1}^4 \left(\hat a_{mni}(\tilde z) + \sum_{l=1}^{n_u} b_{mnil} \vartheta_u^{(0)}(\tilde{z})_l \right) \varphi_{mni}(z)\,.
\end{align*}
Since this must hold for $z$ almost everywhere in the domain, the coefficients must satisfy
\begin{align}
\alpha_{mni}^{(1)}(\tilde z) - \hat a_{mni}(\tilde z) - \sum_{l=1}^{n_u} b_{mnil} \vartheta_u^{(0)}(\tilde{z})=0 \qquad \forall\, m,n\in[0,\dots,k]\,, \ i\in[4]\,. 
\end{align}
For all other time steps, the following parameterization holds:
\begin{align*}
\vartheta_x^{(t+1)}(\tilde z, z) &\approx \sum_{m,n=0}^k \sum_{i=1}^4 \alpha_{mni}^{(t+1)}(\tilde z) \varphi_{mni}(z) \\
 &= \iint \left(\sum_{m,n=0}^k \sum_{i=1}^4 a_{mni} \varphi_i(z-z')\right)\left( \sum_{\hat m,\hat n=0}^k \sum_{\hat i=1}^4 \alpha_{\hat m \hat n \hat i}^{(t)}(\tilde z) \varphi_{\hat m \hat n \hat i}(z') \right) \d z' + b(z)^\top \vartheta_u^{(t)}(\tilde z) \,.
\end{align*}
Since the basis functions are orthogonal for $(m,n,i) \ne (\hat m, \hat n,\hat i)$ for the domain over which the integral is defined, only the $(m,n,i) = (\hat m, \hat n,\hat i)$ terms are nonzero:
\begin{align*}
 \sum_{m,n=0}^k \sum_{i=1}^4 \alpha_{mni}^{(t+1)}(\tilde z) \varphi_{mni}(z) &= \sum_{m,n=0}^k \sum_{i=1}^4  \iint a_{mni}  \varphi_{mni}(z-z')  \alpha_{m  n i}^{(t)}(\tilde z) \varphi_{ mni}(z') \d z' + b(z)^\top \vartheta_u^{(t)}(\tilde z)  \\
 &= \sum_{m,n=0}^k \sum_{i=1}^4 a_{mni} \alpha_{m  n i}^{(t)}(\tilde z) \iint   \varphi_{mni}(z-z')   \varphi_{ mni}(z') \d z' + b(z)^\top \vartheta_u^{(t)}(\tilde z)  \,.
\end{align*}
A character-building computation, which can be done with software such as Mathematica, reduces the integral
\begin{align*}
    \sum_{m,n=0}^k \sum_{i=1}^4 \alpha_{mni}^{(t+1)}(\tilde z) \varphi_{mni}(z) &= 
    \sum_{m,n=0}^k \sum_{i=1}^4 4 \lambda_1 \lambda_2 \bar A_{mn}[i,:] \alpha_{m  n}^{(t)}(\tilde z) \varphi_{mn i}(z)  + b(z)^\top \vartheta_u^{(t)}(\tilde z)  \,,
\end{align*}
where $\bar A_{mn}$, $\alpha_{mn}^{(t)}$ are given by
\begin{align*}
    \bar A_{mn} = \begin{bmatrix}
                    a_{mn4} & a_{mn3} & a_{mn2} & a_{mn1} \\
                    -a_{mn3} & a_{mn4} & -a_{mn1} & a_{mn2} \\
                    -a_{mn2} & -a_{mn1} & a_{mn4} & a_{mn3} \\
                    a_{mn1} & -a_{mn2} & -a_{mn3} & a_{mn4} 
                \end{bmatrix}\,, \qquad 
    \alpha_{mn}^{(t)}(\tilde z) = \begin{bmatrix} 
        \alpha_{mn1}^{(t)}(\tilde z) \\ \alpha_{mn2}^{(t)}(\tilde z) \\ \alpha_{mn3}^{(t)}(\tilde z) \\ \alpha_{mn4}^{(t)}(\tilde z)
                        \end{bmatrix} \,.
\end{align*}
Since this also must hold for $z$ almost everywhere in the domain, the constraint can be written in terms of coefficients
\begin{align}
      \alpha_{mni}^{(t+1)}(\tilde z) -  4 \lambda_1 \lambda_2 \bar A_{mn}[i,:] \alpha_{m  n}^{(t)}(\tilde z)  - \sum_{l=1}^{n_u} b_{mnil} \vartheta_u^{(t)}(\tilde{z})_l  = 0 \qquad \forall\, m,n\in [0,\dots,k]\,,  i\in[4] \,, \tag{2}
\end{align}
The optimization problem we would like to solve, written in terms of the basis functions, is
\begin{align*}
J(\vartheta (\tilde z)) & = \sum_{t=0}^T \left( Q \iint |\vartheta_x^{(t)}(\tilde z,z) |^2 d \tilde z +  R |\vartheta_u^{(t)}(\tilde z)|^2  \right) \\
                     & \approx \sum_{t=0}^T \left( Q \iint \left| \sum_{m,n=0}^k \sum_{i=1}^4 \alpha_{mni}^{(t)} \varphi_{mni}(z) \right |^2 d \tilde z +  R |\vartheta_u^{(t)}(\tilde z)|^2  \right) \,,
\end{align*}
and since the basis functions are orthonormal, the cost can be written in terms of the coefficients $\alpha_{mni}^{(t)}$:
\begin{align*}
J(\vartheta_u(\tilde z),\alpha) \approx \sum_{t=0}^T \sum_{m,n=0}^k \sum_{i=1}^4  \left( Q  (\alpha_{mni}^{(t)})^2  +  R |\vartheta_u^{t,\tau}(\tilde z)|^2  \right) \,.
\end{align*}
The optimization problem that we solve is
\begin{align*}
&\ \argmin_{\vartheta_u(\tilde z),\alpha} J(\vartheta_u(\tilde z),\alpha) \\
&\text{such that } (1),(2) \text{ hold.}
\end{align*}
\printbibliography

@inproceedings{dean2020robust,
  title={Robust guarantees for perception-based control},
  author={Dean, Sarah and Matni, Nikolai and Recht, Benjamin and Ye, Vickie},
  booktitle={Proc. L4DC},
  pages={350--360},
  year={2020},
  organization={PMLR}
}

@article{chen2024robust,
  title={Robust model predictive control with polytopic model uncertainty through system level synthesis},
  author={Chen, Shaoru and Preciado, Victor M and Morari, Manfred and Matni, Nikolai},
  journal={Automatica},
  volume={162},
  pages={111431},
  year={2024},
  publisher={Elsevier}
}

@article{LEEMAN2024_fastSLS,
title = {Fast System Level Synthesis: Robust Model Predictive Control using {Riccati} Recursions},
journal = {IFAC-PapersOnLine},
volume = {58},
number = {18},
pages = {173-180},
year = {2024},
issn = {2405-8963},
doi = {https://doi.org/10.1016/j.ifacol.2024.09.027},
author = {Antoine P. Leeman and Johannes K{\"o}hler and Florian Messerer and Amon Lahr and Moritz Diehl and Melanie N. Zeilinger},
}

@article{diamond2016cvxpy,
  author  = {Steven Diamond and Stephen Boyd},
  title   = {{CVXPY}: {A} {P}ython-embedded modeling language for convex optimization},
  journal = {Journal of Machine Learning Research},
  year    = {2016},
  volume  = {17},
  number  = {83},
  pages   = {1--5},
}

@inproceedings{yu2021localized,
  title={Localized and Distributed ${H}_2$ State Feedback Control},
  author={Yu, Jing and Wang, Yuh-Shyang and Anderson, James},
  booktitle={Proc. ACC},
  pages={2732--2738},
  year={2021},
  organization={IEEE}
}

@book{troltzsch2010optimal,
  title={Optimal control of partial differential equations: theory, methods, and applications},
  author={Tr{\"o}ltzsch, Fredi},
  volume={112},
  year={2010},
  publisher={AMS}
}

@article{anderson2019system,
  title={System level synthesis},
  author={Anderson, James and Doyle, John C and Low, Steven H and Matni, Nikolai},
  journal={Annual Reviews in Control},
  volume={47},
  pages={364--393},
  year={2019},
  publisher={Elsevier}
}

@article{shapiro_windfarm,
author = {Shapiro, Carl R. and Bauweraerts, Pieter and Meyers, Johan and Meneveau, Charles and Gayme, Dennice F.},
title = {Model-based receding horizon control of wind farms for secondary frequency regulation},
journal = {Wind Energy},
volume = {20},
number = {7},
pages = {1261-1275},
doi = {https://doi.org/10.1002/we.2093},
url = {https://onlinelibrary.wiley.com/doi/abs/10.1002/we.2093},
eprint = {https://onlinelibrary.wiley.com/doi/pdf/10.1002/we.2093},
year = {2017}
}

@article{wen_robust_1989,
	title = {Robust adaptive control in {Hilbert} {Space}},
	volume = {143},
	issn = {0022-247X},
	url = {https://www.sciencedirect.com/science/article/pii/0022247X89900255},
	doi = {10.1016/0022-247X(89)90025-5},
	number = {1},
	urldate = {2024-09-25},
	journal = {Journal of Mathematical Analysis and Applications},
	author = {Wen, John Ting-Yung and Balas, Mark J},
	year = {1989},
	pages = {1--26},
}

@article{slemrod_note_1974,
	title = {A {Note} on {Complete} {Controllability} and {Stabilizability} for {Linear} {Control} {Systems} in {Hilbert} {Space}},
	volume = {12},
	issn = {0036-1402},
	url = {http://epubs.siam.org/doi/10.1137/0312038},
	doi = {10.1137/0312038},
	language = {en},
	number = {3},
	urldate = {2024-09-25},
	journal = {SIAM Journal on Control},
	author = {Slemrod, Marshall},
	year = {1974},
	pages = {500--508},
}

@article{gibson_riccati_1979,
	title = {The {Riccati} {Integral} {Equations} for {Optimal} {Control} {Problems} on {Hilbert} {Spaces}},
	volume = {17},
	issn = {0363-0129, 1095-7138},
	url = {http://epubs.siam.org/doi/10.1137/0317039},
	doi = {10.1137/0317039},
	language = {en},
	number = {4},
	urldate = {2024-09-25},
	journal = {SIAM Journal on Control and Optimization},
	author = {Gibson, J. S.},
	year = {1979},
	pages = {537--565},
}

@book{smyshlyaev_adaptive_2010,
	title = {Adaptive {Control} of {Parabolic} {PDEs}},
	copyright = {De Gruyter expressly reserves the right to use all content for commercial text and data mining within the meaning of Section 44b of the German Copyright Act.},
	isbn = {9781400835362},
	url = {https://www.degruyter.com/document/doi/10.1515/9781400835362/html},
	language = {en},
	urldate = {2024-09-25},
	publisher = {Princeton University Press},
	author = {Smyshlyaev, Andrey and Krstic, Miroslav},
	year = {2010},
}

@book{perthame_transport_2007,
	series = {Frontiers in {Mathematics}},
	title = {Transport {Equations} in {Biology}},
	copyright = {http://www.springer.com/tdm},
	isbn = {9783764378417},
	url = {http://link.springer.com/10.1007/978-3-7643-7842-4},
	language = {en},
	urldate = {2024-09-25},
	publisher = {Birkh\"a user},
	author = {Perthame, BenoÃ®t},
	year = {2007},
	doi = {10.1007/978-3-7643-7842-4},
}

@article{bot_admm_2019,
	title = {{ADMM} for monotone operators: convergence analysis and rates},
	volume = {45},
	issn = {1572-9044},
	shorttitle = {{ADMM} for monotone operators},
	url = {https://doi.org/10.1007/s10444-018-9619-3},
	doi = {10.1007/s10444-018-9619-3},
	language = {en},
	number = {1},
	urldate = {2024-09-25},
	journal = {Advances in Comp. Mathematics},
	author = {Bo\c{t}, Radu Ioan and Csetnek, Ern\"o" Robert},
	year = {2019},
	pages = {327--359},
}

@article{wang_separable_2018,
	title = {Separable and {Localized} {System}-{Level} {Synthesis} for {Large}-{Scale} {Systems}},
	volume = {63},
	issn = {1558-2523},
	doi = {10.1109/TAC.2018.2819246},
	number = {12},
	journal = {IEEE TAC},
	author = {Wang, Yuh-Shyang and Matni, Nikolai and Doyle, John C.},
	year = {2018},
	pages = {4234--4249},
}

@article{venkatesh_robust_2000,
	title = {On robust control synthesis and analysis in a {Hilbert} space},
	volume = {39},
	issn = {0167-6911},
	url = {https://www.sciencedirect.com/science/article/pii/S016769119900078X},
	doi = {10.1016/S0167-6911(99)00078-X},
	number = {1},
	urldate = {2024-09-25},
	journal = {Systems \& Control Letters},
	author = {Venkatesh, S. R. and Megretski, A. and Dahleh, M. A.},
	year = {2000},
	pages = {1--12},
}

@article{cattafesta_actuators_2011,
	title = {Actuators for {Active} {Flow} {Control}},
	volume = {43},
	issn = {0066-4189, 1545-4479},
	url = {https://www.annualreviews.org/doi/10.1146/annurev-fluid-122109-160634},
	doi = {10.1146/annurev-fluid-122109-160634},
	language = {en},
	number = {1},
	urldate = {2024-09-26},
	journal = {Annual Review of Fluid Mechanics},
	author = {Cattafesta, Louis N. and Sheplak, Mark},
	year = {2011},
	pages = {247--272},
}

@article{joslin_aircraft_1998,
	title = {Aircraft {Laminar} {Flow} {Control}},
	volume = {30},
	issn = {0066-4189, 1545-4479},
	url = {https://www.annualreviews.org/doi/10.1146/annurev.fluid.30.1.1},
	doi = {10.1146/annurev.fluid.30.1.1},
	language = {en},
	number = {1},
	urldate = {2024-09-26},
	journal = {Annual Review of Fluid Mechanics},
	author = {Joslin, Ronald D.},
	year = {1998},
	pages = {1--29},
}

@book{Hinze2012,
author="Hinze, Michael
and R{\"o}sch, Arnd",
title="Discretization of Optimal Control Problems",
bookTitle="Constrained Optimization and Optimal Control for Partial Differential Equations",
year="2012",
publisher="Springer",
isbn="978-3-0348-0133-1",
doi="10.1007/978-3-0348-0133-1_21",
url="https://doi.org/10.1007/978-3-0348-0133-1_21"
}

@incollection{manzoni_quadratic_2021,
	title = {Quadratic {Control} {Problems} {Governed} by {Linear} {Elliptic} {PDEs}},
	isbn = {9783030772260},
	url = {https://doi.org/10.1007/978-3-030-77226-0_5},
	language = {en},
	urldate = {2024-09-26},
	booktitle = {Optimal {Control} of {Partial} {Differential} {Equations}: {Analysis}, {Approximation}, and {Applications}},
	publisher = {Springer International Publishing},
	author = {Manzoni, Andrea and Quarteroni, Alfio and Salsa, Sandro},
	editor2 = {Manzoni, Andrea and Quarteroni, Alfio and Salsa, Sandro},
	year = {2021},
	doi = {10.1007/978-3-030-77226-0_5},
	pages = {103--165},
}

@inproceedings{ayamou_finite-dimensional_2024,
	title = {Finite-dimensional homogeneous boundary control for a {1D} reaction-diffusion equation},
	url = {https://hal.science/hal-04695713},
	urldate = {2024-09-26},
	booktitle = {Proc. 63rd {IEEE} CDC},
	author = {Ayamou, Mericel and Espitia, Nicolas and Polyakov, Andrey and Fridman, Emilia},
	year = {2024},
}

@article{si_boundary_2018,
	title = {Boundary control for a class of reaction-diffusion systems},
	volume = {15},
	issn = {1751-8520},
	url = {https://doi.org/10.1007/s11633-016-1012-4},
	doi = {10.1007/s11633-016-1012-4},
	language = {en},
	number = {1},
	urldate = {2024-09-26},
	journal = {International Journal of Automation and Computing},
	author = {Si, Yuan-Chao and Xie, Cheng-Kang and Zhao, Na},
	year = {2018},
	pages = {94--102},
}

@article{vazquez_boundary_2019,
	title = {Boundary control and estimation of reaction-diffusion equations on the sphere under revolution symmetry conditions},
	volume = {92},
	issn = {0020-7179, 1366-5820},
	url = {https://www.tandfonline.com/doi/full/10.1080/00207179.2017.1286691},
	doi = {10.1080/00207179.2017.1286691},
	language = {en},
	number = {1},
	urldate = {2024-09-26},
	journal = {Int. Journal of Control},
	author = {Vazquez, Rafael and Krstic, Miroslav},
	year = {2019},
	pages = {2--11},
}

@article{liu_non-commutative_2019,
	title = {Non-commutative discretize-then-optimize algorithms for elliptic {PDE}-constrained optimal control problems},
	volume = {362},
	issn = {0377-0427},
	url = {https://www.sciencedirect.com/science/article/pii/S0377042718304485},
	doi = {10.1016/j.cam.2018.07.028},
	urldate = {2024-09-26},
	journal = {Journal of Comp. and Applied Mathematics},
	author = {Liu, Jun and Wang, Zhu},
	year = {2019},
	pages = {596--613},
}

@ARTICLE{AmoAlonso_MPC,
  author={Amo Alonso, Carmen and Li, Jing Shuang and Anderson, James and Matni, Nikolai},
  journal={IEEE Trans. on Control of Network Systems}, 
  title={Distributed and Localized Model-Predictive Control–Part {I}: Synthesis and Implementation}, 
  year={2023},
  volume={10},
  number={2},
  pages={1058-1068},
  doi={10.1109/TCNS.2022.3219770}}

@article{showalter_hilbert_2009,
	title = {Hilbert {Space} {Methods} for {Partial} {Differential} {Equations}},
	copyright = {Copyright (c) 2023 Ralph E. Showalter},
	issn = {1072-6691},
	url = {https://ejde-ojs-txstate.tdl.org/ejde/article/view/30},
	doi = {10.58997/ejde.mon.01},
	language = {en},
	urldate = {2024-09-26},
	journal = {Electronic Journal of Diff. Equations},
	author = {Showalter, Ralph E.},
	year = {2009},
	pages = {01--214},
}

@inproceedings{conger_characterizing_2024,
	title = {Characterizing {Controllability} and {Observability} for {Systems} with {Locality}, {Communication}, and {Actuation} {Constraints}},
	author = {Conger, Lauren and Li, Yiheng and Mazumdar, Eric and Wierman, Adam},
	year = {2024},
	booktitle = {Proc. 63rd {IEEE} CDC},
}

@article{ascencio_backstepping_2018,
	title = {Backstepping {PDE} {Design}: {A} {Convex} {Optimization} {Approach}},
	volume = {63},
	issn = {1558-2523},
	shorttitle = {Backstepping {PDE} {Design}},
	url = {https://ieeexplore.ieee.org/abstract/document/8052131},
	doi = {10.1109/TAC.2017.2757088},
	number = {7},
	urldate = {2024-10-01},
	journal = {IEEE TAC},
	author = {Ascencio, Pedro and Astolfi, Alessandro and Parisini, Thomas},
	month = jul,
	year = {2018},
	pages = {1943--1958},
}

@phdthesis{jensen_topics_2020,
	type = {Thesis},
	title = {Topics in {Optimal} {Distributed} {Control}},
	school = {University of California Santa Barbara},
	author = {Jensen, Emily},
	year = {2020},
}

@article{bamieh_distributed_2002,
	title = {Distributed control of spatially invariant systems},
	volume = {47},
	issn = {1558-2523},
	url2 = {https://ieeexplore.ieee.org/document/1017552},
	doi2 = {10.1109/TAC.2002.800646},
	number = {7},
	urldate2 = {2024-10-02},
	journal = {IEEE TAC},
	author = {Bamieh, B. and Paganini, F. and Dahleh, M.A.},
	month = jul,
	year = {2002},
	pages = {1091--1107},
}
\end{document}